\documentclass{amsart}
\usepackage{amsmath,amssymb,amsthm,amsfonts,amscd,mathrsfs}
\usepackage[all]{xy}
\usepackage[dvips]{graphicx}

\newcommand{\bdm}{\begin{displaymath}}
\newcommand{\edm}{\end{displaymath}}
\newcommand{\R}{\mathbb{R}}
\newcommand{\Z}{\mathbb{Z}}
\newcommand{\C}{\mathbb{C}}
\newcommand{\co}{\colon\thinspace}
\newcommand{\TC}{\mathbf{TC}}
\newcommand{\TCG}{\mathbf{TC}_G}
\newcommand{\TCK}{\mathbf{TC}_K}
\newcommand{\cO}{{\mathcal O}}

\theoremstyle{definition}
\newtheorem{defn}{Definition}[section]

\newtheorem{exam}[defn]{Example}
\newtheorem{remark}[defn]{Remark}

\theoremstyle{plain}
\newtheorem{thm}[defn]{Theorem}
\newtheorem{prop}[defn]{Proposition}
\newtheorem{lemma}[defn]{Lemma}
\newtheorem{cor}[defn]{Corollary}

\newcommand{\nil}{\mbox{\rm nil}}
\newcommand{\id}{\operatorname{id}}
\newcommand{\cat}{\mbox{\rm cat}}
\newcommand{\gcat}{\mbox{\rm cat}_G}
\newcommand{\secat}{\mbox{\rm secat}\,}
\newcommand{\gsecat}{\mbox{\rm secat}_G\,}

\author{Hellen Colman}
\author{Mark Grant}
\email{hcolman@ccc.edu}
\email{mark.grant@nottingham.ac.uk}
\address{Department of Mathematics, Wright College,
 4300 N.\ Narragansett Avenue, Chicago, IL 60634 USA}
\address{School of Mathematical Sciences, The University of Nottingham,
University Park, Nottingham, NG7 2RD, UK}

\title[Equivariant topological complexity]{Equivariant topological complexity}


\begin{document}

\begin{abstract}
We define and study an equivariant version of Farber's topological complexity for spaces with a given compact group action. This is a special case of the equivariant sectional category of an equivariant map, also defined in this paper. The relationship of these invariants with the equivariant Lusternik-Schnirelmann category is given. Several examples and computations serve to highlight the similarities and differences with the non-equivariant case. We also indicate how the equivariant topological complexity can be used to give estimates of the non-equivariant topological complexity.
\end{abstract}

\maketitle

\section{Introduction}

The sectional category of a map $p\co E\to B$, denoted $\secat(p)$, is the minimum number of open sets needed to cover $B$, on each of which $p$ admits a homotopy section. It was first studied extensively by \v Svarc \cite{S} for fibrations (under the name genus) and later by Berstein and Ganea \cite{BG} for arbitrary maps. The notion of sectional category generalizes the classical (Lusternik-Schnirelmann) category, since $\secat(p)=\cat(B)$ whenever $E$ is contractible and $p$ is surjective. For a general overview of these and other category-type notions, we refer the reader to the survey article of James \cite{Jam} and the book of Cornea-Lupton-Oprea-Tanr\' e \cite{CLOT}.

Further to the classical applications of category to critical point theory, the concept of sectional category has been applied in a variety of settings.
We mention the work of Smale \cite{Sma} and Vassiliev \cite{V} on the complexity of algorithms for solving polynomial equations,
and applications to the theory of embeddings \cite{S}. More recently,  Farber has applied these ideas to the motion planning problem in robotics \cite{Far03,Far04}.
 He defines the {\em topological complexity} of a space $X$, denoted $\TC(X)$, to be the sectional category of the free path fibration on $X$.
The topological complexity is a numerical homotopy invariant which measures the `navigational complexity' of $X$, when viewed as the configuration space of a mechanical system. Along with various related invariants, it has enjoyed much attention in the recent literature (see \cite{GL,BGRT,G} for example).

In this paper we begin a systematic study of the equivariant versions of these notions. For simplicity, we restrict to compact group actions (although most of our results remain true for proper actions). Let $G$ be a compact Hausdorff topological group, and let $p\co E\to B$ be a $G$-map. Then the equivariant sectional category of $p$, denoted $\gsecat(p)$, is the minimum number of invariant open sets needed to cover $B$, on each of which $p$ admits a $G$-homotopy section. If $p$ is a $G$-fibration, this is equivalent to asking for a $G$-equivariant section on each open set in the cover. In the case when the actions are trivial, $\gsecat(p)$ reduces to the ordinary (non-equivariant) sectional category $\secat(p)$.

The equivariant sectional category does not seem to have appeared in the literature until now, although we note below (Corollary \ref{contractible}) that it generalizes the equivariant category, or $G$-category, in many cases of interest. This latter invariant has been extensively studied (see for example \cite{F,Mar,HC}), and gives a lower bound for the number of critical orbits of a $G$-invariant functional on a $G$-manifold. We include a review of some of its properties in Section 3 below, where we also prove product inequalities for equivariant category.

The equivariant topological complexity of a $G$-space $X$, denoted $\TCG(X)$,  is defined in Section 5 to be the equivariant sectional category of the free path fibration $\pi\co X^I\to X\times X$, where $G$ acts diagonally on the product and in the obvious way on paths in $X$. After proving that $\TCG(X)$ is a $G$-homotopy invariant (Theorem \ref{Ghinv}), we give several inequalities relating $\TCG(X)$ to the equivariant and non-equivariant categories and topological complexities of the various fixed point sets. We also show by examples that $\TCG(X)$ can be equal to $\TC(X)$, or at the other extreme, one can be finite and the other infinite (this always happens for example if $X$ is a $G$-manifold which is connected but not $G$-connected). For a group acting on itself by left translations, we show that $\TCG(G)=\cat(G)$, so that category of Lie groups is obtained as a special case of equivariant topological complexity (Theorem \ref{catgrp}). Various other results are given, including a lower bound in terms of equivariant cohomology (Theorem \ref{eqcohom}) and an inequality which bounds the ordinary topological complexity of the fibre space with fibre $X$ associated to a numerable principal $G$-bundle by the product of $\TCG(X)$ and the topological complexity of the base space (Theorem \ref{fred}).

The invariant $\TCG(X)$ has an interpretation in terms of the motion planning problem, when $X$ is viewed as the configuration space of a mechanical system which exhibits $G$ as a group of symmetries. Namely, it is the minimum number of domains of continuity of motion planners in $X$ which preserve the symmetry. Whilst we do not pursue this viewpoint here, it is conceivable that the invariant $\TCG(X)$ may find applications in practical problems of engineering. For more background on the topological approach to motion planning, we refer the reader to \cite{Far06}.

The computation of category and topological complexity in the non-equivariant case are difficult problems which continue to inspire a great deal of research in homotopy theory, and serve to gauge the power of new topological techniques. We believe that the equivariant counterparts of these problems can fill a similar niche in equivariant homotopy theory.

\subsection*{Acknowledgements}
The work of the first author was done during a sabbatical leave granted by the City Colleges of Chicago in the spring of 2012. This paper was finalized during the visit of the first author to the University of Nottingham. She gratefully acknowledges the host institution for its hospitality and the AWM-NSF Travel Grant for its economic support. Both authors thank the anonymous referee for suggestions which greatly improved the exposition of the paper.

\section{Topological Complexity}

We begin by recalling some definitions and fixing some notation. The term {\em fibration} will always refer to a Hurewicz fibration. The sectional category of a fibration was introduced by \v{S}varc  (under the name {\em genus}) and generalized to any map by Berstein and Ganea.




\begin{defn} \cite{BG} The {\em sectional category} of a map $p\co E\to B$, denoted $\secat(p)$,  is the least integer $k$ such that $B$ may be covered by $k$ open sets $\{ U_1,\ldots, U_k\}$ on each of which there exists a map $s\co U_i\to E$ such that $ps\co U_i\to B$ is homotopic to the inclusion $i_{U_i}\co  U_i\hookrightarrow B$. If no such integer exists we set $\secat(p)=\infty$.
\end{defn}

The sets $U_i\subseteq B$ in the above definition are called {\em sectional categorical} for $p$.

\begin{remark} \v{S}varc's original definition \cite{S} of the genus of a fibration $p\co E\to B$ was as the least integer $k$ such that $B$ may be covered by $k$ open sets $\{ U_1,\ldots, U_k\}$ on each of which there exists a local section of $p$, that is to say a map $s\co U_i\to E$ such that $ps=i_{U_i}$. It is easy to see (using the HLP) that this coincides with $\secat(p)$ defined above.
\end{remark}

\begin{remark} In this paper, all our category-type invariants are un-normalized (they are equal to the number of open sets in the cover). For instance, $\secat(p)=1$ if and only if $p$ admits a homotopy section. Thus our definitions exceed by one those in the book \cite{CLOT}, but are in agreement with those used in previous works of the authors, such as \cite{HC,G}.
\end{remark}


Recall that the {\em (Lusternik-Schnirelmann) category} of a space $X$, denoted $\cat(X)$, is the least $k$ such that $X$ may be covered  by $k$ open sets $\{ U_1,\ldots , U_k\}$ such that each inclusion $i_{U_i}\co U_i\hookrightarrow X$ is null-homotopic. The $U_i$ are called {\em categorical} sets.

The sectional category of a surjective fibration $p\co E\to B$  is bounded above by the category of the base, $\secat(p)\le \cat(B)$, and they coincide if the space $E$ is contractible \cite{S}.

Lower bounds for the sectional category can be found using cohomology. If $I$ is an ideal in the commutative ring $R$, the {\em nilpotency} of $I$, denoted $\nil\, I$, is the maximum number of factors in a nonzero product of elements from $I$. Let $H^*$ denote cohomology with coefficients in an arbitrary commutative ring.

\begin{prop}\cite{S}
Let $p\co E\to B$ be a fibration and $p^*\co H^*(B)\to H^*(E)$ be the induced homomorphism.  Then $\secat(p)> \nil\ker p^*$.
\end{prop}


The topological complexity of a space $X$, denoted $\TC(X)$, is a homotopy invariant defined by  Farber \cite{Far03} in order to study the motion planning problem in robotics. We recall now some of its important properties. For more detail we refer the reader to the original papers of Farber \cite{Far03,Far04,Far06}. For any space $X$, let $X^I$ denote the space of paths in $X$ endowed with the compact-open topology. The {\em free path fibration} is the map $\pi\co X^I\to X\times X$ given by $\pi(\gamma)=(\gamma(0), \gamma(1))$. It is surjective if $X$ is path-connected.

\begin{defn} The {\em topological complexity} of a space $X$ is
\[
\TC(X)=\secat(\pi),
\]
the sectional category of the free path fibration $\pi\co X^I\to X\times X$.
\end{defn}

\begin{prop}
If $X$ dominates $Y$, then $\TC(X)\geq \TC(Y)$. In particular, if $X\simeq Y$ then $\TC(X)=\TC(Y)$.
\end{prop}

\begin{prop} For a path-connected space $X$,
$
\cat(X)\le\TC(X)\le\cat(X\times X).
$
\end{prop}

\begin{prop} If $X$ is path-connected and paracompact then
$
\TC(X)\le 2\dim X +1,
$
where $\dim$ denotes the covering dimension.
\end{prop}

\begin{defn} Let $\Bbbk$ be a field. Then cup product defines a homomorphism of rings
\[
\xymatrix{
H^*(X;\Bbbk)\otimes_\Bbbk H^*(X;\Bbbk)\ar[r]^-{\cup} & H^*(X;\Bbbk).
}
\]
The {\em ideal of zero-divisors} $\mathcal{Z}_\Bbbk\subseteq H^*(X;\Bbbk)\otimes_\Bbbk H^*(X;\Bbbk)$ is the kernel of this homomorphism.
\end{defn}
\begin{prop}[Cohomological lower bound] $\TC(X)>\nil \,\mathcal{Z}_\Bbbk$ for any field $\Bbbk$.
\end{prop}

\begin{exam}\label{TCspheres}
The topological complexity of the standard $n$-sphere  is
\[
\TC(S^n) = \left\{ \begin{array}{ll} \infty & (n=0) \\
                                        2 & (n\ge 1\mbox{ odd}) \\
                                     3 & (n\ge 2\mbox{ even}). \end{array} \right.
\]
\end{exam}

\section{Equivariant category}

In this section we recall some definitions and results related to the equivariant (Lusternik-Schnirelmann) category of a $G$-space. We also prove a product inequality for a diagonal action with fixed points, and state the analogous inequality for product actions.

{\em For the remainder of the paper, $G$ will denote a compact Hausdorff topological group acting continuously on a Hausdorff space $X$ on the left}. In this case, we say that $X$ is a {\em $G$-space}. For each $x\in X$ the {\em isotropy group} $G_x=\{h\in G\mid  hx=x\}$ is a closed subgroup of $G$. The set $Gx=\{gx\mid g\in G\}\subseteq X$ is called the {\em orbit} of $x$, and also denoted $\cO (x)$. There is a homeomorphism from the coset space $G/G_x$ to $Gx$, which sends $gG_x$ to $gx$ for each $g\in G$.

The {\em orbit space} $X/G$ is the set of equivalence classes determined by the action, endowed with the quotient topology. Since $G$ is compact and $X$ is Hausdorff, $X/G$ is also Hausdorff, and the {\em orbit map} $p\co X\to X/G$ sending a point to its orbit is both open and closed \cite[Chapter I.3]{tD}.

 If $H$ is a closed subgroup of $G$, then $X^H=\{x\in X|\; hx=x \mbox{ for all }h\in H\}$ is called the {\em $H$-fixed point set} of $X$.

Let $X$ and $Y$ be $G$-spaces.
Two $G$-maps $\phi, \psi\co  X\to Y$ are {\em $G$-homotopic}, written $\phi\simeq_G \psi$, if there is a $G$-map $F\co X \times I \to Y$ with $F_0=\phi$ and $F_1=\psi$, where $G$ acts trivially on $I$ and diagonally on $X\times I$.

 We now begin to discuss the equivariant category of a $G$-space $X$, as studied for instance in \cite{HC,F,Mar}. An open set $U\subseteq X$ is described as {\em invariant} if $gU\subseteq U$ for all $g\in G$.

 \begin{defn} An invariant set $U$ in a $G$-space $X$ is called {\em $G$-categorical} if the inclusion $i_U\co U\to X$ is $G$-homotopic to a map with values in a single orbit.
\end{defn}
\begin{defn} The {\em equivariant category} of a $G$-space $X$, denoted $\gcat(X)$, is the least integer $k$ such that $X$ may be covered by $k$ open sets $\{ U_1,\ldots , U_k\}$, each of which is $G$-categorical.
\end{defn}

\begin{defn} A $G$-space $X$ is said to be {\em $G$-contractible} if $\gcat(X)=1$.
\end{defn}

\begin{exam}\label{circle} Let $G=S^1$ acting freely on $X=S^1$ by rotations. Since the action is transitive we have $\gcat(X)=1$, whilst $\cat(X)=2$. Note that $X$ is $G$-contractible but not contractible.
\end{exam}

\begin{prop}[\cite{F,Mar}]\label{X/G}
When $X$ is a free metrizable $G$-space we have
\[
\gcat(X)=\cat(X/G),
\]
the non-equivariant category of the orbit space. In general, $\gcat(X)\geq\cat(X/G)$.
\end{prop}
The equivariant category of a $G$-space is independent from the category of the space, as the following family of examples illustrates.

\begin{exam} For $n\ge 1$, let $G=S^1\subset\C$ act on the unit sphere $S^{2n-1}\subset\C^{n}$ by complex multiplication in each co-ordinate. Then $\gcat(S^{2n-1}) = \cat(\C P^{n-1}) = n$, whilst $\cat(S^{2n-1})=2$.
\end{exam}

Just like its non-equivariant counterpart, the $G$-category finds applications in critical point theory.

\begin{thm}[\cite{F,Mar}] \label{crit}
Let $M$ be a compact $G$-manifold, and let $f\co M\to \R$ be a smooth $G$-invariant function on $M$. Then $f$ has at least $\gcat(M)$ critical orbits.
\end{thm}

\begin{exam} \label{U(n)}

Let $U(n)$ denote the compact Lie group of $n\times n$-unitary matrices. Then $U(n)$ acts smoothly on itself by conjugation, $A\cdot B = ABA^{-1}$. We can apply Theorem \ref{crit} to obtain an upper bound for $\cat_{U(n)}(U(n))$, as follows.

By diagonalization, two unitary matrices are conjugate if and only if they have the same set of eigenvalues (all of which lie on the unit circle in $\C$). Thus we can define an invariant functional
\[
f\co U(n)\to \R,\qquad f(A) = \sum_{i=1}^n |\lambda_i - 1|,
\]
where $\{\lambda_1, \ldots , \lambda_n\}$ is the set of eigenvalues of $A$. It is easy to see that $f$ is smooth, and that the critical orbits are the conjugacy classes of the matrices $\operatorname{diag}(-1,\ldots , -1,1,\ldots, 1)$. There are precisely $n+1$ such orbits. We therefore have $\cat_{U(n)}(U(n))\le n+1$ by Theorem \ref{crit}.

In Example \ref{conjugation} below we will see that $\cat_{U(n)}(U(n))= n+1$, by relating it to the notion of equivariant topological complexity. \qed
\end{exam}

 We now give an equivariant version of the product inequality for category. Our treatment is based on \cite[Theorem 1.37]{CLOT}, which in turn is based on that of Fox \cite[Theorem 9]{Fox}. In particular, our proof relies on a notion of categorical sequence.

\begin{defn} A {\em $G$-categorical sequence} in $X$ of length $k$ is a nested sequence $A_0\subseteq A_1\subseteq\cdots\subseteq A_k=X$ with $A_0=\emptyset$ and the property that each difference $A_i-A_{i-1}$ is invariant and is contained in some $G$-categorical open set $U_i$. (Note that we do not require the $A_i$ to themselves be invariant.)
\end{defn}
\begin{lemma} A $G$-space $X$ has a $G$-categorical sequence of length $k$ if and only if $\gcat(X)\leq k$. \qed
\end{lemma}

Just as in the non-equivariant case, we also need some separation and connectedness conditions. Recall that a space $X$ is called {\em completely normal} if whenever $A,B\subseteq X$ such that $\overline{A}\cap B = \emptyset = A\cap\overline{B}$, then $A$ and $B$ have disjoint open neighbourhoods in $X$. For example, metric spaces and CW-complexes are completely normal.

\begin{defn}
A $G$-space $X$ is called {\em $G$-completely normal} if whenever $A,B\subseteq X$ are invariant sets such that $\overline{A}\cap B = \emptyset = A\cap\overline{B}$, then $A$ and $B$ have disjoint open invariant neighbourhoods in $X$.
\end{defn}

\begin{lemma}
If $X$ is a completely normal $G$-space, then $X$ is $G$-completely normal.
\end{lemma}

\begin{proof}
It is well known that if $X$ is completely normal and $G$ is compact, then $X/G$ is completely normal. Thus it suffices to prove that complete normality of $X/G$ implies $G$-complete normality of $X$. This is an exercise in general topology, using the orbit map $p\co X\to X/G$, and is left to the reader.
\end{proof}

\begin{defn}
A $G$-space $X$ is said to be {\em $G$-connected} if
the $H$-fixed point set $X^H$ is path-connected for every closed subgroup $H$ of
$G$.
\end{defn}

\begin{lemma}[Conservation of isotropy]\label{4}
Let $X$ be a $G$-connected $G$-space, and let $x, y \in X$ such that $G_x\subseteq G_y$. Then there exists a $G$-homotopy $F\co \cO(x)\times I\to X$ such that $F_0=i_{\cO(x)}$ and $F_1(\cO(x))\subset \cO(y)$.
\end{lemma}
\begin{proof}
Let $H=G_x$. Then $x, y\in X^H$ since $G_{x}\subseteq G_{y}$. Consider a path $\alpha\co I\to X^H$ joining $x$ and $y$. Then $H\subseteq G_{\alpha(t)}$ for all $t\in I$. Define a homotopy $F\co  {G/{G_x}}\times I \to X$ given by $F(gG_x, t)=g\alpha(t)$. We have that $F$ is well defined, is equivariant and is a homotopy of the inclusion into the orbit $\cO(y)$.
\end{proof}

\begin{thm}\label{productI}
  Let $X$ and $Y$ be $G$-connected $G$-spaces such that $X\times Y$ is completely normal. If $X^G\neq \emptyset$ or $Y^G\neq\emptyset$, then
\[
\gcat(X\times Y)\leq\gcat(X)+\gcat(Y)-1,
\]
where $X\times Y$ is given the diagonal $G$-action.
\end{thm}

\begin{proof}
Suppose that $\gcat(X)=n$ with $G$-categorical sequence $\{A_0,A_1,\ldots , A_n\}$ and $\gcat(Y)=m$ with $G$-categorical sequence $\{B_0,B_1,\ldots , B_m\}$. Denote by $U_i\subseteq X$ the open $G$-categorical set containing $A_i-A_{i-1}$, and by $W_j\subseteq Y$ the open $G$-categorical set containing $B_j-B_{j-1}$. Suppose for concreteness that $X^G\neq\emptyset$. By Lemma \ref{4} we may assume that the inclusions $i_{U_i}\co U_i\to X$ are all $G$-homotopic into $\cO (x_0) = \{x_0\}$, where $x_0\in X^G$ is some fixed point. Each inclusion $i_{W_j}\co W_j\to Y$ is $G$-homotopic into $\cO(y_j)$ for some $y_j\in Y$. It follows that $U_i\times W_j\subseteq X\times Y$ is $G$-homotopic into $\cO(x_0,y_j) = \{x_0\}\times\cO(y_j)$, and hence these products are all $G$-categorical in $X\times Y$.

Define subsets of $X\times Y$ by
\[
C_0=\emptyset,\qquad C_k = \bigcup_{r=1}^{k} A_r\times B_{k+1-r}\quad(1\leq k\leq n+m-1),
\]
 where we set $A_i=\emptyset$ for $i>n$ and $B_j=\emptyset$ for $j>m$. We claim that $\{C_0,\ldots, C_{n+m-1}\}$ is a $G$-categorical sequence for $X\times Y$.

 The proof of this claim proceeds by analogy with the non-equivariant case \cite[Theorem 1.37]{CLOT}, using the $G$-complete normality of $X\times Y$. Therefore we omit the details.
\end{proof}

We remark that a similar result (with a similar proof) was given by Cicorta\c s \cite[Proposition 3.2]{Cic}. There the assumption on the existence of fixed points was omitted, however, leading to counter-examples. For example, let $G=S^1$ acting on $X=Y=S^1$ by rotations as in Example \ref{circle}. Then $\gcat(X\times Y) = \cat((S^1\times S^1)/S^1) = \cat(S^1) = 2$, whilst $\gcat(X) + \gcat(Y) - 1 = 1$. The problem is that an orbit of the diagonal action is not necessarily a product of orbits.

The hypothesis on fixed point sets can be dropped when considering more general product actions. Let $K$ be another compact Hausdorff group. Then the product of a $G$-space $X$ and a $K$-space $Y$ becomes a $G\times K$-space in an obvious way. The orbits of this action are the products of orbits, and one easily obtains the following result.

\begin{thm}\label{productII}
Let $X$ be a path-connected $G$-space and $Y$ be a path-connected $K$-space, such that $X\times Y$ is completely normal. Then
\[
\cat_{G\times K}(X\times Y)\leq \gcat(X) + \cat_K(Y)-1.
\]
\end{thm}

\section{Equivariant sectional category}




In this section we generalize the notion of sectional category to the equivariant setting.

\begin{defn} The {\em equivariant sectional category} of a $G$-map $p\co E\to B$, denoted $\gsecat(p)$, is the least integer $k$ such that $B$ may be covered by $k$ invariant open sets $\{ U_1,\ldots, U_k\}$ on each of which there exists a $G$-map $s\co U_i\to E$ such that  $ps\simeq_G i_{U_i}\co  {U_i}\hookrightarrow B$.
\end{defn}
The sets $U_i\subseteq B$ in the above will be called {\em $G$-sectional categorical} for $p$.

\begin{prop}\label{secatfibr} If $p\co E\to B$ is a $G$-fibration, then $\gsecat(p)\leq k$ if and only if $B$ may be covered by $k$ invariant open sets $\{ U_1,\ldots, U_k\}$ on each of which there exists a {\em local $G$-section}, that is a $G$-map $\sigma_i\co U_i\to E$ such that  $p\sigma_i = i_{U_i}\co  {U_i}\hookrightarrow B$.
\end{prop}

\begin{proof}
This is analogous to the non-equivariant case, using the $G$-HLP \cite[page 53]{tD}.
\end{proof}
Next we observe that equivariant sectional category of $G$-fibrations cannot increase under taking pullbacks (compare \cite[Proposition 7]{S}).

\begin{prop}\label{pullback}
Let $p\co E\to B$ be a $G$-fibration and $f\co A\to B$ be a $G$-map. The pullback $q\co A\times_B E \to A$ of $p$ along $f$ satisfies
$\gsecat(q)\leq\gsecat(p)$.
 \end{prop}
 \begin{proof}
 Given an invariant open set $U\subseteq B$ with $G$-section $s\co U\to E$ of $p$, one obtains a $G$-section $\sigma\co f^{-1}(U)\to A\times_B E$ of $q$ by setting $\sigma(a) = (a,sf(a))$.
 \end{proof}

We now study conditions under which the equivariant category of a $G$-space $B$ is an upper bound for the equivariant sectional category of a $G$-map $p\co E\to B$.

\begin{prop}[Equivariant version of the connectivity condition]\label{5}
Let $p\co E\to B$ be a $G$-map.
If $B$ is $G$-connected and $E^G\not = \emptyset$, then $\gsecat(p)\le \gcat(B)$.
\end{prop}

\begin{proof}
Let $U$ be a $G$-categorical set for $B$ and let $F\co U\times I\to B$ be the $G$-homotopy such that $F_0=i_U$ and $F_1=c$ with $c(U)\subseteq \cO(x_0)$. Choose an $e\in E^G$, and let $b=p(e)$. By Lemma \ref{4}, since $B$ is $G$-connected and $G_{x_0}\subseteq G_b=G$ we have that there exists  a $G$-homotopy $\Phi\co \cO(x_0)\times I\to B$ such that $\Phi_0=i_{\cO(x_0)}$ and $\Phi_1(\cO(x_0))\subseteq \cO(b)=\{b\}$.
Consider $s\co  U\to E$ given by $s(x)=e$ for all $x\in U$. The map $s$ is equivariant and the composition of the homotopies $F$ and $\Phi$ provides a homotopy from $i_U$ to $ps$.
%
\end{proof}

\begin{prop}[Equivariant version of the surjectivity condition]\label{6}
Let $p\co E\to B$ be a $G$-map.
If $p(E^H)=B^H$ for all closed subgroups $H$ of $G$, then $\gsecat(p)\le \gcat(B)$.
\end{prop}

\begin{proof}
Let $U$ be a $G$-categorical set for $B$ and let $F\co U\times I\to B$ be the $G$-homotopy such that $F_0=i_U$ and $F_1=c$ with $c(U)\subseteq \cO(x_0)$.
Let $H=G_{x_0}$, then $x_0\in B^H=p(E^H)$. Therefore there exists $z_0\in E^H$ such that $p(z_0)=x_0$. Define $s\co  U\to E$ by $s(x)=g z_0$ if $c(x)=g x_0$.
The proof that $s$ is equivariant follows from the fact that  $c$ is equivariant. Moreover, $ps$ is $G$-homotopic to the inclusion since $ps(x)=p(g z_0)=g p(z_0)=g x_0=c(x)$ and $c\simeq_G i_U$.
\end{proof}






Finally in this section, we investigate equivariant sectional category in the case when the total space is $G$-contractible.

\begin{prop}\label{Gcontract} Let $p\co E\to B$ be a $G$-map.
If $E$ is a $G$-contractible space then $\gcat (B) \le \gsecat (p)$.
\end{prop}

\proof
Let $U$ be a $G$-sectional categorical set for $p$ and $s\co  U\to E$ be a $G$-map such that $ps\simeq_G i_U$. We have that the identity map $\id_E$ is $G$-homotopic to a $G$-map $c$ whose image is contained in a single orbit $\cO(x)$.
Then $pcs\simeq_G ps\simeq_Gi_U$ and  $pcs(U)$ is contained in $\cO(p(x))$. Thus $U$ is $G$-categorical in $B$.
\qed

\begin{cor} \label{contractible} Let $p\co E\to B$ be a $G$-map and $E$ be a $G$-contractible space. If $B$ is $G$-connected and $E^G\not = \emptyset$ or if $p(E^H)=B^H$ for all closed subgroups $H$ of $G$, then $\gsecat (p)= \gcat (B)$.
\end{cor}

In particular, if $X$ is a $G$-connected space with a fixed point $x\in X^G$, then the inclusion $p\co \{x\}\to X$ is a $G$-map with $\gsecat (p)= \gcat (X)$.

\section{Equivariant topological complexity}

When $X$ is a $G$-space, the free path fibration $\pi\co X^I\to X\times X$ is a $G$-fibration with respect to the actions
$$G\times X^I\to X^I, \qquad G\times X\times X\to X\times X,$$
$$g(\gamma)(t)=g(\gamma(t)),\qquad g(x,y)=(gx,gy).$$
The verification of this fact is straightforward; one may use a $G$-equivariant version of \cite[Theorem 2.8.2]{Spa}.

\begin{defn} The {\em equivariant topological complexity} of the
  $G$-space $X$, denoted $\TCG(X)$, is defined as the equivariant sectional category of the free path fibration $\pi\co X^I\to X\times X$. That is, $$\TCG(X)=\gsecat(\pi\co X^I\to X\times X).$$
  \end{defn}

  In other words, the equivariant topological complexity is the least integer $k$ such that
  $X\times X$ may be covered by $k$ invariant open sets
  $\{U_1,\ldots ,U_k\}$, on each of which there is a $G$-equivariant
  map $s_i\co U_i\to X^I$ such that $\pi s_i\simeq_G i_{U_i}\co  U_i\hookrightarrow X\times X$. (Since $\pi$ is a $G$-fibration, this is equivalent to requiring $\sigma_i\co U_i\to X^I$ such that $\pi \sigma_i = i_{U_i}$.) If no such integer exists then we set $\TCG(X)=\infty$.

  We first show that equivariant topological complexity is a $G$-homotopy invariant. Let $X$ and $Y$ be $G$-spaces. We say that $X$ {\em $G$-dominates} $Y$ if there exist $G$-maps $\phi\co X\to Y$ and $\psi\co Y\to X$ such that $\phi\psi\simeq_G \id_Y$. If in addition $\psi\phi\simeq_G \id_X$, then $\phi$ and $\psi$ are {\em $G$-homotopy equivalences}, and $X$ and $Y$ are {\em $G$-homotopy equivalent}, written $X\simeq_G Y$.

\begin{thm}\label{Ghinv}
If $X$ $G$-dominates $Y$ then $\TCG(X)\ge \TCG(Y)$. In particular, if $X\simeq_G Y$ then $\TCG(X)=\TCG(Y)$.
\end{thm}
\begin{proof}
Let $\phi\co X\to Y$ and $\psi\co Y\to X$ be $G$-maps such that $\phi\psi\simeq_G \id_Y$, and let $U\subseteq X\times X$ be $G$-sectional categorical for $\pi_X\co X^I\to X\times X$. Hence there exists a $G$-map $s\co U\to X^I$ such that $\pi_X s\simeq_G i_U\co U\hookrightarrow X\times X$.

We will show that $V=(\psi\times \psi)^{-1}U\subseteq Y\times Y$ is $G$-sectional categorical for $\pi_Y\co Y^I\to Y\times Y$. Denote by $\overline{(\psi\times\psi)}\co V\to U$ the map obtained by restricting the domain and range of $(\psi\times\psi)$. Let $\sigma\co V\to Y^I$ be the composition $\sigma = \widetilde\phi \circ s \circ \overline{(\psi\times\psi)}$, where $\widetilde\phi\co X^I\to Y^I$ is the map induced by $\phi$. Then
$$
 \pi_Y\sigma
            = (\phi\times\phi)\pi_X s \overline{(\psi\times\psi)}
            \simeq_G (\phi\times\phi) i_U \overline{(\psi\times\psi)}
            = (\phi\times\phi)(\psi\times \psi) i_V
           \simeq_G  i_V,
$$
 hence $V$ is $G$-sectional categorical for $\pi_Y$.

 Now if $\{U_1,\ldots , U_k\}$ is a $G$-sectional categorical cover of $X\times X$, then $\{V_1,\ldots, V_k\}$ defined as above is a $G$-sectional categorical cover of $Y\times Y$. This proves the first statement, and the second follows immediately.
\end{proof}

It is obvious that $\TC(X)\leq\TCG(X)$ for any $G$-space $X$. More generally we have the following.

\begin{prop}\label{fixed}
Let $X$ be a $G$-space, and let $H$ and $K$ be closed subgroups of $G$ such that $X^H$ is $K$-invariant. Then $\TCK(X^H)\le\TCG (X)$.
\end{prop}
\begin{proof}
Let $U\subseteq X\times X$ be $G$-sectional categorical for $\pi\co X^I\to X\times X$, and let $\sigma\co U\to X^I$ be a $G$-map such that $\pi \sigma =  i_{U}\co  U\hookrightarrow X\times X$.

  Define $V=U\cap (X^H\times X^H)\subseteq X^H\times X^H$, and note that $V$ is $K$-invariant.
 Since $\sigma$ is $G$-equivariant it takes $H$-fixed points to $H$-fixed points, and so restricts to a $K$-equivariant map $\sigma_V\co V\to (X^H)^I$. It is clear that $\pi\sigma_V=i_V$ and therefore $V$ is $K$-sectional categorical for $\pi|_{X^H}\co (X^H)^I\to X^H\times X^H$.
\end{proof}

\begin{cor}\label{subgroups} Let $X$ be a $G$-space. Then
\begin{enumerate}
\item $\TC(X^H)\le\TCG (X)$ for all closed subgroups $H$  of $G$.
\item $\TCK(X)\le\TCG (X)$ for all closed subgroups $K$  of $G$.
\end{enumerate}
\end{cor}

\begin{cor} If $X$ is not $G$-connected, then $\TCG(X)=\infty$.
\end{cor}
\begin{proof}  Let $H$ be a closed subgroup of $G$ such that $X^H$ is not path-connected. Then $\TC(X^H)=\infty$ and the result follows from Corollary \ref{subgroups} (1).
\end{proof}

The next three results describe the basic relationship of equivariant topological complexity with equivariant category.

\begin{prop}\label{upper} If $X$ is $G$-connected, then $\TCG(X)\le \gcat(X\times X)$.
\end{prop}
\begin{proof}  Given a closed subgroup $H$ of $G$, $(X\times X)^H=X^H\times X^H$ is path-connected since $X$ is $G$-connected. Then the map
$\pi|_{X^H}\co  (X^H)^I \to X^H\times X^H$ is surjective and the result follows from Proposition \ref{6}.
\end{proof}

\begin{prop}\label{lower}  If $X$ is $G$-connected, and $H=G_z\subseteq G$ is the isotropy group of some $z\in X$, then
$\cat_H(X)\le\TCG(X)$.
\end{prop}
\begin{proof}
We have $\TC_H(X)\le \TCG(X)$ by Corollary \ref{subgroups} (2). When we pull back $\pi\co X^I\to X\times X$ along the $H$-equivariant map
\[
j\co X\to X\times X,\qquad j(x)=(z,x)
\]
we obtain an $H$-fibration $p\co PX\to X$ whose total space $PX=\{ \gamma\co I\to X \mid \gamma(0)=z\}$ is $H$-contractible. Therefore $\cat_H(X)\le \secat_H(p)\le \TC_H(X)$ by Propositions \ref{Gcontract} and \ref{pullback}.
\end{proof}

\begin{cor}\label{catfixed} Let $X$ be a $G$-connected $G$-space with $X^G\not=\emptyset$. Then
\begin{enumerate}
\item  $\gcat(X)\le\TCG(X)\le 2\gcat(X)-1$.
\item  $\TCG(X)=1$ if and only if $X$ is $G$-contractible.
\end{enumerate}
\end{cor}

\begin{proof}
Part (1) follows directly from Propositions \ref{upper} and \ref{lower}, and Theorem \ref{productI}. Part (2) follows from part (1), since by definition $X$ is $G$-contractible if and only if $\gcat(X)=1$.
\end{proof}

We now turn to examples.

\begin{exam}[Spheres under reflection] \label{spheres} For $n\geq 1$, let $X=S^n\subseteq \R^{n+1}$ with the group $G=\Z_2$ acting by the reflection given by multiplication by $-1$ in the last co-ordinate.

When $n=1$ the fixed point set $X^G = \{(1,0),(-1,0)\}$ is disconnected, and so $\TCG(X) = \infty$ in this case.

When $n\geq 2$ the fixed point set $X^G = S^{n-1}$ is the equatorial sphere, hence is connected. Note that $\gcat(X)=2$ in this case ($X$ is clearly not $G$-contractible since the orbits are discrete; we leave it to the reader to construct a cover by two $G$-categorical open sets). Therefore we have
\[
\TCG(X)\le 2\gcat(X)-1 = 3
\]
by Corollary \ref{catfixed} (1). When $n$ is even we have $3=\TC(X)\le \TCG(X)$, and when $n$ is odd we have $3=\TC(X^G)\le \TCG(X)$ by Corollary \ref{subgroups} (1). We have therefore shown that
\[
\TCG(S^n) = \left\{ \begin{array}{rl} \infty & (n=1)\\
                                        3    & (n\ge 2). \end{array}\right.
\]
 \end{exam}

\begin{exam}
If $X=S^1$ is the circle with $G=S^1$ acting on $X$ by
rotations, then $\TCG(X)=2$.
For the usual motion planner on $S^1$ with two local rules is equivariant
with respect to rotations. So $\TCG(X)\leq 2$, and $\TCG(X)\geq
\TC(X)=2$.
\end{exam}

 Generalizing the previous example, we have the following result which shows that the category of a connected group is a particular instance of its equivariant topological complexity.

\begin{thm} \label{catgrp} Let $G$ be a connected metrizable group acting on itself by left translation. Then $\TCG(G)=\cat(G)$.
\end{thm}
\begin{proof} Since the diagonal action of $G$ on $G\times G$ is free, we have
\begin{align*}
\cat(G) & \leq \TCG(G) \qquad\mbox{(Proposition \ref{lower})}\\
       & \leq\gcat(G\times G) \qquad\mbox{(Proposition \ref{upper})} \\
       & = \cat\left((G\times G)/G\right) \qquad\mbox{(Proposition \ref{X/G})}\\
       & = \cat(G),
\end{align*}
where at the last step we have made use of the fact that there is a homeomorphism $(G\times G)/G$ to $G$ sending $[g,h]$ to $g^{-1}h$.
\end{proof}

This illustrates the importance of the fixed point set in determining equivariant topological complexity. In particular it shows that $\TCG(X)$ can be arbitrarily large even when $X$ is $G$-contractible, as long as $X$ has no fixed points. By contrast, a $G$-contractible space with fixed points has $\TCG(X)=1$, by Corollary \ref{catfixed} (2).

Next we give an equivariant version of a result of Farber \cite[Lemma 8.2]{Far04} which states that the topological complexity of a connected topological group equals its category.

\begin{prop}\label{topgroups}
Let $X$ be a topological group. Assume that $G$ acts on $X$ by topological group homomorphisms, and that $X$ is $G$-connected. Then $\TCG(X)=\gcat(X)$.
\end{prop}

\begin{proof}
We first note that the identity element $e\in X$ is a fixed point, since $G$ acts by group homomorphisms. It follows from Proposition \ref{lower} that $\gcat(X)\le \TCG(X)$.

Now suppose that $\gcat(X)=k$. Let $\{ U_1,\ldots , U_k\}$ be a $G$-categorical open cover of $X$. By Lemma \ref{4}, for each $i$ we can find a $G$-homotopy $F\co U_i\times I\to X$ such that $F_0=i_{U_i}$ and $F_1(U_i)\subseteq \cO(e)=\{e\}$.

Set $V_i = \{ (a,b)\mid ab^{-1}\in U_i\}\subseteq X\times X$. Since multiplication and inversion are continuous and $G$-equivariant, the $V_i$ are open and $G$-invariant. Define a $G$-section on $V_i$ by setting $s_i(a,b)(t) = F(ab^{-1},t)b$; this is easily checked as being $G$-equivariant. Hence $\{V_1,\ldots , V_k\}$ forms a $G$-sectional categorical open cover, and $\TCG(X)\le k = \gcat(X)$. This completes the proof.
\end{proof}

\begin{exam} \label{conjugation} Let $G$ be a connected Lie group, acting on itself by conjugation $g\cdot a = gag^{-1}$. Note that $G$ acts by homomorphisms. The fixed point set $G^H$ of a closed subgroup $H\subseteq G$ is the centralizer $C_G(H) = \{ g\in G\mid ga=ag\mbox{ for all } a\in H\}$. Hence $G$ is $G$-connected if and only if the centralizer of every closed subgroup is connected (this holds for example if $G$ is $U(n)$ for $n\geq 1$, or a product of such). In this case, Proposition \ref{topgroups} applies and gives
\[
\TCG(G)=\gcat(G).
\]
We now look at the case $G=U(n)$ in more detail. By results of Farber \cite[Lemma 8.2]{Far04} and Singhof \cite[Theorem 1(b)]{Si}, we have
\[
\TC_{U(n)}(U(n))\ge \TC(U(n)) = \cat(U(n)) = n+1.
\]
On the other hand, Example \ref{U(n)} gives $\TC_{U(n)}(U(n))=\cat_{U(n)}(U(n))\le n+1$. Therefore,
\[
\TC_{U(n)}(U(n)) =\cat_{U(n)}(U(n)) =  n+1.
\]
\end{exam}

Next we give a cohomological lower bound for $\TCG(X)$, using equivariant cohomology theory. We use the following equivariant generalization of \cite[Lemma 18.1]{Far06}, whose proof is routine.

\begin{lemma}
An invariant open subset $U\subseteq
X\times X$ is $G$-sectional categorical with respect to $\pi\co X^I\to X\times X$ if and only if the inclusion $i_U\co U\hookrightarrow X\times X$
is $G$-homotopic to a map with values in the diagonal
$\triangle X\subseteq X\times X$. \qed
\end{lemma}

Let $EG\to BG$ denote a universal principal $G$-bundle, and $X^h_G=EG\times_G X$ the corresponding homotopy orbit space of $X$. Denote by
$H^*_G(X)=H^*(X^h_G)$ the Borel $G$-equivariant cohomology of $X$, with coefficients in an arbitrary commutative ring. Note that
the diagonal map $\triangle\co X\to X\times X$ is equivariant, and
hence induces a map $\triangle_G\co X^h_G\to (X\times X)^h_G$.

\begin{thm}[Cohomological lower bound] \label{eqcohom}
Suppose that there are cohomology classes $z_1,\ldots ,z_k\in
H^*_G(X\times X)$ such that $0= \triangle_G^*(z_i)\in H^*_G(X)$ for all $i$ and the
product $z_1\cdots z_k$ is non-zero. Then $\TCG(X)>k$.
\end{thm}
\begin{proof}
Suppose $\TCG(X)\leq k$, and let $\{U_1,\ldots , U_k\}$ be a $G$-sectional categorical open cover. Since the inclusion $i_{U_i}\co U_i\hookrightarrow X\times X$ factors through $\triangle\co X\to X\times X$ up to $G$-homotopy, it follows that the restriction $(i_{U_i})^*\co H^*_G(X\times X)\to H^*_G(U_i)$ maps $z_i$ to zero, and hence
$z_i$ is in the image of $H^*_G(X\times X,U_i)\to H^*_G(X\times X)$
for each $i=1,\ldots , k$. By naturality of cup products, it then follows that the product $z_1\cdots z_k$ is
zero.
\end{proof}

We do not currently know of any examples of $G$-connected spaces $X$ where the lower bound for $\TCG(X)$ given by Theorem \ref{eqcohom} improves on the non-equivariant lower bound $\nil\ker (\cup)<\TC(X)\le\TCG(X)$ given by the zero-divisors cup-length.

Finally in this section, we prove a result which relates equivariant and non-equivariant topological complexity, and give an example indicating that the former may be useful in estimating the latter.

\begin{thm}\label{fred} Let $X$ be a $G$-space, and let $E\to B=E/G$ be a numerable principal $G$-bundle. Then
$$\TC(X_G)\leq\TCG(X)\TC(B),$$
where $X_G=E\times_G X$ denotes the associated fibre space with fibre $X$.
\end{thm}
\begin{proof}
Suppose that $\TC(B)=k$ by a cover $B\times
B=U_1\cup\cdots\cup U_k$ and $\TCG(X)=\ell$ by a cover $X\times
X=W_1\cup\cdots\cup W_\ell$ by $G$-invariant open sets admitting $G$-sections $s_j\co W_j\to X^I$. Our aim is to cover $X_G\times X_G$ by $k\ell$ open sets on which the map $\Pi\co (X_G)^I\to X_G\times X_G$ admits a homotopy section. We have a strictly commuting diagram
\[
\xymatrix{
   E\times_G(X^I)\ar[r] \ar[d]_{E\times_G\pi} \ar[dr]^p &  (X_G)^I \ar[d]^{\Pi}    \\
 E\times_G(X\times X) \ar@{^{(}->}[r]^-\iota \ar[d] & X_G\times X_G \ar[d]^{\rho\times\rho} \\
 B \ar@{^{(}->}[r]^\triangle & B\times B
}
\]
whose bottom square is a pullback. Here $\iota$ is the map sending $[e,x,x']$ to $([e,x],[e,x'])$ and $\rho\co X_G\to B$ is the fibration sending $[e,x]$ to $[e]$.

To prove the proposition it suffices to cover $X_G\times X_G$ by $k\ell$ open sets on which the map $p$ given by $p[e,\gamma] = ([e,\gamma(0)],[e,\gamma(1)])$ admits a homotopy section.

By Lemma 18.1 of \cite{Far06} each
inclusion $U_i\hookrightarrow B\times B$ is homotopic to a map with
values in the diagonal $\triangle B\subseteq B\times B$. Setting $V_i=(\rho\times \rho)^{-1}(U_i)\subseteq X_G\times X_G$ and applying the homotopy lifting property of the map $\rho\times \rho$, we obtain a homotopy from the inclusion $V_i\hookrightarrow X_G\times X_G$ to a map $H_i\co V_i\to X_G\times X_G$ with values in $E\times_G(X\times X)$, and hence by restricting the range a map $h_i\co V_i\to E\times_G(X\times X)$.

Since the sets $W_j\subseteq X\times X$ form an invariant open cover, the sets $E\times_G W_j\subseteq E\times_G(X\times X)$ form an open cover. The equivariant sections $s_j\co W_j\to X^I$ give rise to sections $\sigma_j:=E\times_G s_j\co E\times_G W_j\to E\times_G (X^I)$ such that $(E\times_G \pi)\sigma_j = i_{E\times_G W_j}\co E\times_G W_j\hookrightarrow E\times_G(X\times X)$.

We now have a cover of $X_G\times X_G$ by the $k\ell$ open sets
\[
\Omega_{ij} = h_i^{-1}(E\times_G W_j),\qquad i=1,\ldots, k,\qquad j=1,\ldots,\ell,
\]
and candidate sections $\zeta_{ij}=\sigma_j\circ h_i|_{\Omega_{ij}}\co \Omega_{ij}\to E\times_G(X^I)$ of the map $p$. These are in fact homotopy sections, since \[
p\zeta_{ij} = p\sigma_j h_i|_{\Omega_{ij}} = \iota(E\times_G \pi)\sigma_j h_i|_{\Omega_{ij}} = \iota i_{E\times_G W_j} h_i|_{\Omega_{ij}}= H_i|_{\Omega_{ij}} \simeq i_{\Omega_{ij}}.
\]
\end{proof}

\begin{exam}[{Compare \cite[Example 5.8]{G}}]
Let $K^{n+1}$ be the `$(n+1)$-dimensional Klein bottle'. This is the mapping torus of the involution $S^n\to S^n$ given by reflection in the last co-ordinate. Note that $K^2$ is the usual Klein bottle.

Letting $E=S^1$ with free $\Z_2$-action given by the antipodal map, we see that  $K^{n+1}=E\times_{\mathbb{Z}_2} S^n$. If $n\ge 2$ we have
$\TC_{\Z_2}(S^n)=3$ by Example \ref{spheres}, and $\TC(B)=\TC(S^1)=2$. Then Theorem \ref{fred} gives $\TC(K)\leq
3\cdot 2=6$. However Theorem \ref{fred} says nothing in the $n=1$ case, since $\TC_{\Z_2}(S^1)=\infty$.
\end{exam}

\end{document}